\documentclass[12pt,reqno]{amsart}
\usepackage{amssymb,amsmath,tabularx,mathrsfs,mathbbol,yfonts,upgreek}
\usepackage{amsthm,verbatim,comment}
\usepackage{geometry}
\usepackage{stmaryrd}
\usepackage{paralist}
\usepackage[all]{xy}
\DeclareSymbolFontAlphabet{\mathbb}{AMSb}
\DeclareSymbolFontAlphabet{\mathbbl}{bbold}

\newtheorem{thm}{Theorem}[section]
\newtheorem{lem}[thm]{Lemma}
\newtheorem{prop}[thm]{Proposition}

\theoremstyle{definition}

\newtheorem{rem}[thm]{Remark}
\theoremstyle{remarks}
\newtheorem*{rem*}{Remarks}

\renewcommand{\mod}{\mathrm{mod}\ }

\newcommand{\sym}[1]{\mathfrak{S}_{#1}}

\newcommand{\Y}{\mathscr{Y}}
\newcommand{\C}{\mathscr{C}}
\newcommand{\NN}{\mathbb{N}}

\newcommand{\M}{{\pmb{\mathscr{M}}}}

\newcommand{\Z}{\mathbb{Z}}

\renewcommand{\P}{\mathscr{P}}

\newcommand{\cont}{\text{{\tiny $\#$}}}

\newcommand{\W}{\mathscr{W}}
\newcommand{\mbf}[1]{\boldsymbol{#1}}

\numberwithin{equation}{section}

\begin{document}
\title{A note on the signature representations of the symmetric groups}

\author{Kay Jin Lim}
\address[K. J. Lim]{Division of Mathematical Sciences, Nanyang Technological University, SPMS-PAP-03-01, 21 Nanyang Link, Singapore 637371.}
\email{limkj@ntu.edu.sg}
\thanks{The first author is supported by Singapore MOE Tier 2 AcRF MOE2015-T2-2-003.}

\author{Jialin Wang}
\address[J. Wang]{Division of Mathematical Sciences, Nanyang Technological University, SPMS-PAP-03-01, 21 Nanyang Link, Singapore 637371.}
\email{wangjl@ntu.edu.sg}

\begin{abstract} For a partition $\lambda$ and a prime $p$, we prove a necessary and sufficient condition for there exists a composition $\delta$ such that $\delta$ can be obtained from $\lambda$ after rearrangement and all the partial sums of $\delta$ are not divisible by $p$. To demonstrate why we are interested in the question, we compute some signed $p$-Kostka numbers.
\end{abstract}	


\maketitle

\section{Preliminary}

Let $p$ be a prime. In the representation theory of the symmetric groups over a field of characteristic $p$, the Young permutation and Young modules play a central role. The Young modules are the indecomposable summands of the Young permutation modules up to isomorphism and labelled by partitions (see \cite{James}). In \cite{Donkin}, Donkin generalised these objects and obtained the signed Young modules as the indecomposable summands of the signed Young permutation modules up to isomorphism and they are labelled by certain bipartitions. In his paper, the listing modules were also obtained as a generalisation of the tilting modules for the Schur algebras and they are, up to isomorphism, the indecomposable summands of the tensor products of certain symmetric and exterior powers. The original constructions of the Young and signed Young modules used the representation theories of the Schur algebras and superalgebras respectively (as in \cite{James,Donkin}). Constructions of these objects using only the representation theory of the symmetric groups can be found in \cite{Erdmann,GLOW}. The multiplicities of the signed Young modules as direct summands of signed Young permutation modules are known as signed $p$-Kostka numbers. They generalise the classical $p$-Kostka numbers. It is an open problem to determined all signed $p$-Kostka numbers.

Suppose further that $p$ is an odd prime. Let $\lambda$ be a partition and $c_\lambda^{(p)}$ be the number of compositions $\delta$ such that $\delta$ can be obtained from $\lambda$ after rearrangement and all the partial sums of $\delta$ are not divisible by $p$. In \cite[Corollary 4.6]{Lim}, in particular, the first author showed that, in the Green ring of the symmetric groups (respectively, Schur algebras) over a field of characteristic $p$, the signature representation of a symmetric group (respectively, exterior power of certain natural module of a Schur algebra) can be written as a linear combination of the signed Young permutation modules (respectively, mixed powers) labelled by bipartitions of the form $(\lambda|p(s))$ with coefficients $c^{(p)}_\lambda$ up to signs. However, it appears to be difficult to give a closed formula for the coefficients. In Theorem \ref{T: 1} of Section \ref{S: iff}, we give a necessary and sufficient condition for $c^{(p)}_\lambda\neq 0$.

The explicit $p$-Kostka numbers are known for very few cases. For two-part partitions, we refer the reader to \cite{AHenke}. For $p=2$ and hook partitions, we refer the reader to \cite[Proposition 2.18]{O'Donovan}. In Section \ref{S: Kostka}, we calculate some signed $p$-Kostka numbers. It demonstrates why we are interested in the numbers $c^{(p)}_\lambda$.

\bigskip

We now begin by fixing the notation we need throughout.

Let $\NN_0$ be the set of non-negative integers and let $n\in\NN_0$. A composition $\delta$ of $n$ is a sequence of positive integers $(\delta_1,\ldots,\delta_s)$ such that $\sum_{i=1}^s\delta_i=n$. In this case, we write $\ell(\delta)=s$ and $n=|\lambda|$. By convention, the unique composition of 0 is denoted as $\varnothing$ and $\ell(\varnothing)=0$. Let $\eta$ be another composition. The sum $\delta+\eta$ is defined as componentwise summation and, if $m\in \NN_0$, $m\delta$ is the componentwise multiplication of $\delta$ by $m$. We also define the concatenation \[\delta\cont\eta=(\delta_1,\ldots,\delta_s,\eta_1,\ldots,\eta_t)\] if $t=\ell(\eta)$. For each $1\leq j\leq s$, we write \[\delta^+_j=\sum_{i=1}^j\delta_i\] for the partial sum and $\delta^+_s=n$. For each positive integer $d$, the number of parts of $\delta$ equal to $d$ is denoted as $n_d(\delta)$, i.e., \[n_d(\delta)=|\{i:\text{$1\leq i\leq s$ and $\delta_i=d$}\}|.\] The composition $\delta$ is called a partition if $\delta_1\geq \cdots\geq \delta_s$. Let $\P(n)$ be the set of all partitions of $n$. For two compositions $\delta$ and $\eta$, we write $\delta\cup\eta$ for the partition obtained by rearranging $\delta\cont\eta$.

Let $\lambda$ be a partition of $n$. The partition $\lambda$ can be written uniquely as a $p$-adic expansion $\lambda=\sum_{i=0}^\infty p^i\lambda(i)$ where each $\lambda(i)$ is a $p$-restricted partition, i.e., the differences between successive parts of $\lambda(i)$ (including the length of the last part) are strictly less than $p$. For example, $\lambda(0)$ is the partition obtained from $\lambda$ by removing all horizontal $p$-hooks successively.

Let $q$ be a positive integer and $\C(\lambda)$ be the set consisting of all compositions which can be rearranged to $\lambda$. A composition $\delta$ is called $q'$-cumulative if $q\nmid \delta^+_j$ for all $1\leq j\leq \ell(\delta)$. We write $c^{(q)}_\lambda$ for the number of compositions $\delta\in \C(\lambda)$ such that $\delta$ is $q'$-cumulative. Clearly, the number $c^{(q)}_\lambda$ depends only on the parts of $\lambda$ modulo $q$ and $c_\lambda^{(q)}=0$ if $q\mid n$. For each composition $\delta$ and $0\leq j\leq q-1$, let \[r_j(\delta)=\sum_{i\equiv j(\mod q)}n_i(\delta)\] and $r(\delta)=(r_1(\delta),\ldots,r_{q-1}(\delta))$ and let \[\lambda!_q=\frac{\prod_{i=0}^{q-1}r_i(\lambda)!}{\prod_{d\geq 1}n_d(\lambda)!}{\ell(\lambda)-1\choose r_0(\lambda)}\in\NN_0.\] For any $\mathbf{r}=(r_1,\ldots,r_{q-1})\in \NN_0^{q-1}$, we write \[\W^{(q)}_\mathbf{r}=\{\delta\in \C(\mu):\text{$\delta$ is $q'$-cumulative}\}\] where $\mu=((q-1)^{r_{q-1}},\ldots,1^{r_1})$.

\medskip

Let $\C^2(n)$ be the set of all pairs of compositions $(\delta|\eta)$ such that $|\delta|+|\eta|=n$. We write $\unrhd$ for the dominance order on $\P^2(n)$ the subset of $\C^2(n)$ consisting of pairs of partitions (see, for example, \cite[\S3.2]{DL}).  For each $(\lambda|p\mu)\in\P^2(n)$ and $(\alpha|\beta)\in\C^2(n)$, we have the signed Young module $Y(\lambda|p\mu)$ and signed Young permutation module $M(\alpha|\beta)$ (see \cite[\S2.3]{Donkin}). The indecomposable summands of the signed Young permutation modules are the signed Young modules. Notice that $M(\delta|\eta)\cong M(\alpha|\beta)$ if $\delta,\eta$ are rearrangements of $\alpha,\beta$ respectively and $M(\alpha|\beta\cont(1))\cong M(\alpha\cont(1)|\beta)$. We write $k_{(\alpha|\beta),(\lambda|p\mu)}$ for the signed $p$-Kostka number defined as the multiplicity of $Y(\lambda|p\mu)$ as a direct summand of $M(\alpha|\beta)$ up to isomorphism. The following property is crucial.

\begin{thm}[{\cite[2.3(8)]{Donkin}}] Let $(\alpha|\beta),(\lambda|p\mu)\in \P^2(n)$. Then $k_{(\alpha|\beta),(\lambda|p\mu)}=0$ unless $(\lambda|p\mu)\unrhd (\alpha|\beta)$ and  $k_{(\lambda|p\mu),(\lambda|p\mu)}=1$.
\end{thm}

If $\lambda\in\P(n)$ and $\alpha\in\C(n)$, the classical Young module $Y^\lambda$ and Young permutation module $M^\alpha$ satisfy $Y^\lambda\cong Y(\lambda|\varnothing)$ and $M^\alpha\cong M(\alpha|\varnothing)$. Also, the dominance order on $\P^2(n)$ restricts to the usual dominance order on $\P(n)$ under the identification $\lambda=(\lambda|\varnothing)$. In this case, we write $k_{\alpha,\lambda}=k_{(\alpha|\varnothing),(\lambda|\varnothing)}$ for the classical $p$-Kostka numbers.


\section{Necessary and sufficient condition for $c^{(p)}_\lambda\neq 0$}\label{S: iff}

The main aim in this section is to prove Theorem \ref{T: 1}. We begin with an easy proposition.

\begin{prop}\label{P: 1} Let $q$ be a positive integer and $\lambda$ be a partition. Then \[c^{(q)}_\lambda=\left |\W^{(q)}_{r(\lambda)}\right |\lambda!_q.\] In particular, we have $c^{(q)}_\lambda\neq 0$ if and only if $\W^{(q)}_{r(\lambda)}\neq \emptyset$.
\end{prop}
\begin{proof} We only check the converse of the final assertion. If $\W^{(q)}_{r(\lambda)}\neq\emptyset$ then $r_i(\lambda)\neq 0$ for some $1\leq i\leq \ell(\lambda)$. So $\ell(\lambda)-1\geq r_0(\lambda)$ and hence ${\ell(\lambda)-1 \choose r_0(\lambda)}\neq 0$, i.e., $\lambda!_q\neq 0$.
\end{proof}

For example, $c^{(1)}_\lambda=0$ for all $\lambda$ (including when $\lambda=\varnothing$). When $q=2$, $\W^{(2)}_{r(\lambda)}\neq\emptyset$ if and only if $r(\lambda)=(1)$, i.e., $\lambda$ has exactly one part with odd size and the rest with even sizes, and, in this case, $|\W^{(2)}_{(1)}|=1$. Therefore,  \[c^{(2)}_\lambda=\left \{\begin{array}{ll}\lambda!_2&\text{if $r(\lambda)=(1)$,}\\ 0&\text{otherwise.}\end{array}\right .\] When $q=3$, $\W^{(3)}_{r(\lambda)}\neq\emptyset$ if and only if $|r_1(\lambda)-r_2(\lambda)|=1,2$. In this case, it is not difficult to see that $|\W^{(3)}_{r(\lambda)}|=1$. Therefore \[c^{(3)}_\lambda=\left \{\begin{array}{ll} \lambda!_3&\text{if $|r_1(\lambda)-r_2(\lambda)|\in\{1,2\}$,}\\ 0&\text{otherwise.}\end{array}\right .\] A closed formula for general $q$ and $\lambda$ appeared to be difficult to obtain.

To prove our Theorem \ref{T: 1}, we will need the following two lemmas. We begin with some notations.

Let $\mathbf{r}=(r_1,\ldots,r_{q-1})\in \NN_0^{q-1}$. We set \[\|\mathbf{r}\|=\sum_{i=1}^{q-1}ir_i\] which is the size of any composition belonging to the set $\W^{(q)}_{\mathbf{r}}$ so that $q\nmid \|\mathbf{r}\|$ if $\W^{(q)}_\mathbf{r}\neq\emptyset$. Also, set
\begin{align*}
  &|\mathbf{r}|_q=(q-1)+\sum_{i=2}^{q-1}(q-i)r_i,\\
  &\max \mathbf{r}=\max\{r_1,\ldots,r_{q-1}\}.
\end{align*}

\begin{lem}\label{L: 1} Let $q$ be a positive integer, $\mathbf{r}=(r_1,\ldots,r_{q-1})\in\NN_0^{q-1}$ and suppose that $r_1=\max\mathbf{r}$. Then $\W^{(q)}_{\mathbf{r}}\neq \emptyset$ if and only if
\begin{enumerate}
  \item [(i)] $q\nmid \|\mathbf{r}\|$, and
  \item [(ii)] $r_1 \leq |\mathbf{r}|_q$.
\end{enumerate}
\end{lem}
\begin{proof} We argue by induction on $|\mathbf{r}|_q$ and by definition $|\mathbf{r}|_q \geq q-1$. If $|\mathbf{r}|_q=q-1$ then $r_i=0$ for all $2\leq i\leq q-1$. In this case, $\W^{(q)}_\mathbf{r}\neq \emptyset$ if and only if $\W^{(q)}_\mathbf{r}=\{(1^{r_1})\}$ and $0< r_1\leq q-1$, and that is if and only if $q\nmid \|\mathbf{r}\|=r_1$, and $r_1=\max\mathbf{r}\leq |\mathbf{r}|_q=q-1$. Fix a positive integer $N\geq q$. Suppose now that the equivalent statement in the lemma holds true for any $\mathbf{r}\in\NN_0^{q-1}$ such that $r_1=\max\mathbf{r}$ and $|\mathbf{r}|_q<N$. Let $\mathbf{r}=(r_1,\ldots,r_{q-1})\in\NN_0^{q-1}$, $|\mathbf{r}|_q=N$ and $r_1=\max\mathbf{r}$.

Assume that $\W^{(q)}_{\mathbf{r}}\neq \emptyset$ (so part (i) is satisfied). Suppose on the contrary that $r_1=\max\mathbf{r}>|\mathbf{r}|_q$. Let $\delta\in\W^{(q)}_\mathbf{r}$, $s=\ell(\delta)$ and $c$ be maximum such that $\delta_c=b>1$ and $\delta_i=1$ for all $c+1\leq i\leq s$. Since $\delta$ is $q'$-cumulative, we have $s-c<q-1$ (otherwise, $\delta^+_{c+d}= \delta^+_c+d\equiv 0(\mod q)$ for some $1\leq d\leq q-1\leq s-c$). Consider the following two cases.
\begin{enumerate}
  \item [(A)] Suppose that $s-c<q-b$. Let $\delta'$ be the composition such that \[\delta=\delta'\cont(b,1^{s-c}).\] Clearly, $\delta'$ is $q'$-cumulative and hence $\delta'\in\W^{(q)}_{\mathbf{r}'}$ for some $\mathbf{r}'=(r_1',\ldots,r_{q-1}')$. Notice that \[r_1'=r_1-(s-c)>|\mathbf{r}|_q-(q-b)=|\mathbf{r}'|_q > r_j'\] for any $2\leq j\leq q-1$.
  \item [(B)] Suppose that $q-b\leq s-c$. Let $\eta$ be the composition such that \[\delta=\eta\cont(b,1^{q-b})\cont (1^{(s-c)-(q-b)})\] and let $\delta'=\eta\cont (1^{(s-c)-(q-b)})$. Notice that \[\delta'^+_j=\left \{\begin{array}{ll}\delta^+_{j-(q-b+1)}-q&\text{if $j>\ell(\eta)$,}\\ \delta^+_j&\text{otherwise.}\end{array}\right.\] So $\delta'$ is $q'$-cumulative and hence $\delta'\in \W^{(q)}_{\mathbf{r}'}$ for some $\mathbf{r}'=(r_1',\ldots,r_{q-1}')$. Also, \[r'_1=r_1-(q-b)>|\mathbf{r}|_q-(q-b)=|\mathbf{r}'|_q > r_j'\] for any $2\leq j\leq q-1$.
\end{enumerate} In both cases, we have $q\nmid \|\mathbf{r}'\|$ and $r_1'=\max\mathbf{r}'>|\mathbf{r}'|_q<|\mathbf{r}|_q=N$ but yet $\W^{(q)}_{\mathbf{r}'}\neq\emptyset$. This contradicts to our induction hypothesis.

Conversely, assume that both parts (i) and (ii) in the statement hold. In particular, $q>1$. Consider the following two cases.
\begin{enumerate}
  \item [(A)] Suppose there exists $2\leq b\leq q-1$ such that $r_b>0$ and $b\not\equiv \|\mathbf{r}\|(\mod q)$. Let $\mathbf{r}'=(r_1',\ldots,r_{q-1}')$ where $r_j'=r_j$ if $j\neq b$ and $r_b'=r_b-1$. Then $r_1'=\max\mathbf{r}'$, $|\mathbf{r'}|_q=|\mathbf{r}_q|-(q-b)<N$ and \[\|\mathbf{r}'\|=\|\mathbf{r}\|-b\not\equiv 0(\mod q).\] By induction hypothesis, there exists $\delta'\in\W^{(q)}_{\mathbf{r}'}$. It is easy to check that $\delta'\cont(b)\in\W^{(q)}_{\mathbf{r}}$.
  \item [(B)] Suppose that, for any $2\leq b\leq q-1$ with $r_b>0$, we have $b\equiv \|\mathbf{r}\|(\mod q)$. Since $|\mathbf{r}|_q\geq q$, there exists a unique $2\leq b\leq q-1$ such that $r_b>0$. By assumption, $b\equiv \|\mathbf{r}\|(\mod q)$. We further consider two cases.
      \begin{enumerate}
        \item [(a)] Suppose that $r_1\leq |\mathbf{r}|_q-(q-b)$. Let $\mathbf{r}'=(r_1',\ldots,r_{q-1}')$ such that $r_i'=0$ for all $i\neq 1,b$ and $r_i'=r_i-1$ if $i=1,b$. Then
      \begin{align*}
            &\|\mathbf{r}'\|=r_1'+br_b'=\|\mathbf{r}\|-1-b\equiv -1(\mod q),\\
            &\max\mathbf{r}'=r_1'\leq |\mathbf{r}|_q-(q-b)-1< |\mathbf{r}'|_q<N.
      \end{align*} Since $q>2$, we have $-1\not\equiv 0(\mod q)$ and hence $q\nmid \|\mathbf{r}'\|$.  By induction hypothesis, there exists $\delta'\in\W^{(q)}_{\mathbf{r}'}$. Define $\delta=\delta'\cont(b,1)$. Notice that $\delta^+_{\ell(\delta')+1}\equiv b-1(\mod q)$ and $\delta^+_{\ell(\delta')+2}\equiv b(\mod q)$.
        \item [(b)] Suppose that $|\mathbf{r}|_q-(q-b)<r_1$. Let $s=r_1-(q-b)(r_b-1)-(q-1)$. By assumption, $0<s\leq q-b$. Let $\eta=(b,1^{q-b})$ and \[\delta=(1^{q-1})\cont\underbrace{\eta\cont\cdots\cont\eta}_{\text{$(r_b-1)$ times}}\cont (b,1^s).\]
      \end{enumerate} In both cases, it is easy to check that $\delta$ is $q'$-cumulative and hence $\delta\in\W^{(q)}_\mathbf{r}$.
\end{enumerate}
\end{proof}

To state the next lemma, we introduce a notation. Suppose that $a$ is multiplicatively invertible in $\Z/q\Z$ and $\mathbf{r}=(r_1,\ldots,r_{q-1})$. We write \[{}^a\mathbf{r}=(r_1',\ldots,r_{q-1}')\] where $r_j'=r_i$ if $j\equiv ai(\mod q)$ for any $1\leq j\leq q-1$. So ${}^a\mathbf{r}$ is obtained from $\mathbf{r}$ by a permutation determined by $a$.

\begin{lem}\label{L: 2} Let $q$ be a positive integer, $a$ be multiplicatively invertible in $\Z/q\Z$ and $\mathbf{r}=(r_1,\ldots,r_{q-1})\in\NN_0^{q-1}$. Then $|\W^{(q)}_{\mathbf{r}}|=|\W^{(q)}_{{}^a\mathbf{r}}|$.
\end{lem}
\begin{proof} Define $\phi:\W^{(q)}_\mathbf{r}\to\W^{(q)}_{{}^a\mathbf{r}}$ as, for any $\delta=(\delta_1,\ldots,\delta_s)\in \W^{(q)}_\mathbf{r}$, $\phi(\delta)=\delta'=(\delta'_1,\ldots,\delta'_s)$ where $1\leq \delta'_t\leq q-1$ and $\delta'_t\equiv a\delta_t(\mod q)$ for all $1\leq t\leq s$. Notice that, for each $1\leq j\leq q-1$, the number of parts of $\delta'$ with size $j$ is precisely the number of parts of $\delta$ with size $i$ where $j\equiv  ai(\mod q)$ and $\delta'^+_t\equiv a\delta^+_t(\mod q)$. So $\phi$ is well-defined. Since $a$ is invertible in $\Z/q\Z$, $\phi$ is invertible and hence we have $|\W^{(q)}_{\mathbf{r}}|=|\W^{(q)}_{{}^a\mathbf{r}}|$.
\end{proof}

We are now ready to state and prove our main theorem.

\begin{thm}\label{T: 1} Let $p$ be a prime number, $\lambda$ be a partition and $r(\lambda)=(r_1,\ldots,r_{p-1})$. Then $c^{(p)}_\lambda\neq 0$ if and only if
\begin{enumerate}
  \item [(i)] $p\nmid |\lambda|$, and
  \item [(ii)] for some $r_a=\max r(\lambda)$, $r_a \leq |{}^br(\lambda)|_p$ where $ab\equiv 1(\mod p)$.
\end{enumerate}
\end{thm}
\begin{proof} By Proposition \ref{P: 1} and Lemma \ref{L: 2}, we have $c^{(p)}_\lambda\neq 0$ if and only if $\W^{(p)}_{r(\lambda)}\neq \emptyset$ if and only if $\W^{(p)}_{{}^br(\lambda)}\neq \emptyset$ for some $1\leq b\leq p-1$. Let $1\leq a,b\leq p-1$ be such that $r_a=\max r(\lambda)$, $ab\equiv 1(\mod p)$ and ${}^br(\lambda)=\mathbf{r}'=(r_1',\ldots,r_{p-1}')$. Notice that $r_1'=r_a=\max r(\lambda)=\max \mathbf{r}'$. By Lemma \ref{L: 1}, $\W^{(p)}_{{}^br(\lambda)}=\W^{(p)}_{\mathbf{r}'}\neq \emptyset$ if and only if
\begin{enumerate}
  \item [(i)] $p\nmid \|\mathbf{r}'\|$, and
  \item [(ii)] $r'_1 \leq |\mathbf{r}'|_p$.
\end{enumerate} Notice that \begin{align*}
\|\mathbf{r}'\|=\sum_{j=1}^{p-1}jr_j'\equiv\sum^{p-1}_{i=1}ibr_i=b\sum^{p-1}_{i=1}ir_i&=b\sum_{i=0}^{p-1}i\sum_{j\equiv i(\mod p)}n_j(\lambda)\\
&\equiv b\sum_{i=0}^{p-1}\sum_{j\equiv i(\mod p)}jn_j(\lambda)=b|\lambda|(\mod p).
\end{align*} Therefore, $p\nmid \|\mathbf{r}'\|$ is equivalent to $p\nmid |\lambda|$. The proof is now complete.
\end{proof}

We end this section with the following remark.

\begin{rem} Keep the notation as in Theorem \ref{T: 1} and suppose that there is another $1\leq c\leq p-1$ such that $c\neq a$ and $r_a=r_c=\max r(\lambda)$. Then \[|{}^br(\lambda)|_p=(p-1)+\sum^{p-1}_{i=2}(p-i)r_i'\geq (p-j)r_j'=(p-j)r_c\geq r_a\] where $j\equiv bc(\mod p)$. If $p\nmid |\lambda|$, by Theorem \ref{T: 1}, $c^{(p)}_\lambda\neq 0$. In other words, as long as $p\nmid |\lambda|$ and $r(\lambda)$ attains its maximum at least 2 distinct places, we have $c^{(p)}_\lambda\neq 0$.
\end{rem}

\section{Some explicit computation of signed $p$-Kostka numbers}\label{S: Kostka}

Fix an odd prime $p$. In this section, we compute some explicit signed $p$-Kostka numbers and decompose the Young permutation module $M^{(n-2,1,1)}$. The proofs of the statements in this section require notions and notations which have not been discussed earlier in this paper and are not required elsewhere. As such, we only refer the reader to the necessary backgrounds in the proofs. Throughout, we use the convention that if 0 appears in a component of a composition (for some reasons) we simply delete that component. For examples, $(0)=\varnothing$ and $(0,1^r)=(1^r)$.

We begin with the following two lemmas.


\begin{lem}\label{L: simplify} Let $m,k\in\NN_0$  suppose that $k=bp+r$ where $0\leq r\leq p-1$.
\begin{enumerate}
  \item [(i)] In the Green ring, \[[M((m)|(k))]=\sum_{j=0}^b\sum_{\xi\in\P(k-jp)}(-1)^{|\xi|-\ell(\xi)} c_\xi^{(p)}[M((m)\cup\xi|p(j))],\] here, $[-]$ denotes the isomorphism class of the respective module.
  \item [(ii)] Let $n=m+k$. We have the isomorphism
  \begin{align*}
  &M((m)|(k))\\
  \cong &\left\{\begin{array}{ll} Y((m,1^r)|p(b))&\text{if either $p\mid n$ or $k\in\{0,n\}$,}\\ Y((m,1^r)|p(b))\oplus Y((m+1,1^{r-1})|p(b))&\text{if $p\nmid n$ and $r\geq 1$,}\\ Y((m)|p(b))\oplus Y((m+1,1^{p-1})|p(b-1))&\text{if $p\nmid n$, $k\neq 0$ and $r= 0$.}\end{array}\right .
  \end{align*}
\end{enumerate}
\end{lem}
\begin{proof} Let $\alpha=(m)$ and $\beta=(k)$. For part (i), the set $V((k);(\xi,p\mu))$ in \cite[Notation 4.1(ix)]{Lim} is empty unless $p\mu=p(j)=(k-|\xi|)$ for some $\xi\in\P(k-jp)$ and $0\leq j\leq b$. In this case, \[V((k);(\xi,p(j)))=\{(\xi)\}\] and hence $c_{(k);(\xi,p(j))}^{(p)}=(-1)^{|\xi|-\ell(\xi)}c_\xi^{(p)}$. The result now follows from \cite[Corollary 4.6(i)]{Lim}.

For part (ii), the cases $p\mid n$ and $k=n$ have been obtained in \cite[Proposition 7.1]{GLOW} and the case $k=0$ is trivial. Now assume $p\nmid n$ and $k\neq 0$. By the Littlewood-Richardson rule and Nakayama conjecture (\cite{Brauer,Robinson}), we have $M((m)|(k))\cong S^{(m,1^{k})}\oplus S^{(m+1,1^{k-1})}$. Since we assume $p$ is odd, by \cite{Peel}, any Specht module $S^{(n-\ell,1^\ell)}$ labelled by the hook $(n-\ell,1^\ell)$ is irreducible. By \cite[Theorem 5.1]{DL}, $S^{(n-\ell,1^\ell)}$ is isomorphic to the signed Young module $Y((n-\ell,1^s)|p(c))$ where $\ell=cp+s$ and $0\leq s\leq p-1$.
\end{proof}

A direct application of the version of signed Klyachko's formula introduced in \cite{GL} yields the following lemma. However, for completeness, we give a proof for this easy case.

\begin{lem}\label{L: Kostka reduction} Let $\lambda$ be a partition of $n$ such that $\lambda=\lambda(0)+p(a)$ for some $a\in \NN_0$, let $b,j\in\NN_0$, let $\alpha$ be a composition of $n-jp$ such that $0\leq \alpha_i\leq p-1$ for all $i\geq 2$ and let $\alpha=(\alpha_1)\cont\overline{\alpha}$. Then \[k_{(\alpha|p(b+j)),(\lambda|p(b))}=\left \{\begin{array}{ll} k_{((\alpha_1-ap)\cont\overline{\alpha}|p(j)),(\lambda(0)|\varnothing)}&\text{if $\alpha_1\geq ap$,}\\ 0&\text{otherwise.}\end{array}\right .\] Here, we use the convention $k_{(\varnothing|\varnothing),(\varnothing|\varnothing)}=1$.
\end{lem}
\begin{proof} The proof uses various results in \cite{GLOW}. Suppose first that $\beta=p(b+j)$, $p\mu=p(b)$ and, $a=\sum_{i\geq 0} a_ip^i$ and $b=\sum_{i\geq 0} b_ip^i$ are the $p$-adic sums of $a$ and $b$ respectively. Let $n_0=|\lambda(0)|=n-ap$ and $n_i=a_{i-1}+b_{i-1}$ for $i\geq 1$. Then $\rho=(1^{n_0},p^{n_1},\ldots)$ in \cite[Corollary 5.2]{GLOW} and the set $\Lambda:=\Lambda((\alpha|\beta)|\rho)$ (see \cite[Notation 3.8]{GLOW}) consists of all pairs of tuples of compositions \[(\mbf{\gamma}|\mbf{\delta})=((c_0)\cont\overline{\alpha},(c_1),\ldots|(d_0),(d_1),\ldots)\] such that $\alpha_1=\sum_{i\geq 0} c_ip^i$, $(b+j)p=\sum_{i\geq 0} d_ip^i$, $c_0+|\overline{\alpha}|+d_0=n-ap$ and $c_i+d_i=n_i$ for all $i\geq 1$. By \cite[Corollary 5.2]{GLOW}, we have \begin{align*}
&k_{(\alpha|p(b+j)),(\lambda|p(b))}\\
=&\sum_{(\mbf{\gamma}|\mbf{\delta})\in\Lambda} [W_1((c_0)\cont\overline{\alpha}|(d_0)):Y(\lambda(0)|\varnothing)]\cdot \prod_{i\geq 1} [W_{p^i}((c_i)|(d_i)):Q_{p^i}((a_{i-1})|(b_{i-1}))],
\end{align*} here, $[W:Q]$ denotes the multiplicity of an indecomposable module $Q$ as a direct summand of a module $W$ up to isomorphism (Krull-Schmidt Theorem applies here). We claim that, for $i\geq 1$, $[W_{p^i}((c_i)|(d_i)):Q_{p^i}((a_{i-1})|(b_{i-1}))]$ is $1$ if $c_i=a_{i-1}$ (and hence $d_i=b_{i-1}$) and $0$ otherwise. Once we have proved this, we have $c_0=\alpha_1-ap$ and hence $d_0=jp$. Notice that this is a contradiction unless $\alpha_1-ap\geq 0$. In this case, the module $W_1((\alpha_1-ap)\cont\overline{\alpha}|p(j))$ is $M((\alpha_1-ap)\cont\overline{\alpha}|p(j))$ as defined in \cite[Definition 3.6]{GLOW}. In the other case when $\alpha_1<ap$, we deduce that $\Lambda=\emptyset$ and therefore $k_{(\alpha|p(b+j)),(\lambda|p(b))}=0$. This proves our desired result. We should now prove the claim.

Let $P_{p^i}$ be a Sylow $p$-subgroup of $\sym{p^i}$, $N_{p^i}$ be the normaliser of $P_{p^i}$ in $\sym{p^i}$ and $\mathfrak{A}_{p^i}$ be the alternating subgroup of $\sym{p^i}$. Since $i\geq 1$ and $p$ is odd, the normaliser $N_{\mathfrak{A}_{p^i}}(P_{p^i})$ acts trivially on both $W_{p^i}((c_i)|(d_i))$ and $Q_{p^i}((a_{i-1})|(b_{i-1}))$. Therefore, using \cite[Lemma 2.1]{GLOW}, we have \begin{align}\label{Eq: 3}
[W_{p^i}((c_i)|(d_i)):Q_{p^i}((a_{i-1})|(b_{i-1}))]=[\overline{W}_{p^i}((c_i)|(d_i)):\overline{Q}((a_{i-1})|(b_{i-1}))]
\end{align} (see \cite[Definitions 3.6, 4.8 and Lemma 4.9]{GLOW}. By \cite[Proposition 4.5]{GLOW}, Equation \ref{Eq: 3} is equal to zero unless $a_{i-1}=c_i$ and $b_{i-1}=d_i$. In this case, $W_{p^i}((a_{i-1})|(b_{i-1}))\cong Q_{p^i}((a_{i-1})|(b_{i-1}))$ and therefore Equation \ref{Eq: 3} is equal 1.
\end{proof}

Our first result describes the decomposition of the Young permutation module labelled by the `first' 3-part partition $(n-2,1,1)$.

\begin{prop} Let $n\geq 3$ and $r$ be the remainder of $n$ modulo $p$. Then \[M^{(n-2,1,1)}\cong \left \{\begin{array}{ll} Y^{(n-1,1)}\oplus Y^{(n-2,2)}\oplus Y^{(n-2,1^2)}&\text{if $r=0$,}\\ 2\cdot Y^{(n-1,1)}\oplus Y^{(n-2,2)}\oplus Y^{(n-2,1^2)}&\text{if $r=1$,}\\Y^{(n)}\oplus Y^{(n-1,1)}\oplus Y^{(n-2,2)}\oplus Y^{(n-2,1^2)}&\text{if $r=2$,}\\ Y^{(n)}\oplus 2\cdot Y^{(n-1,1)}\oplus Y^{(n-2,2)}\oplus Y^{(n-2,1^2)}&\text{if $r\geq 3$.}\end{array}\right .\]
\end{prop}
\begin{proof} Let $k=2$ and $m=n-2$ in Lemma \ref{L: simplify}. We have
\begin{align*} -[M((n-2,2)|\varnothing)]+[M((n-2,1,1)|\varnothing)]&=[M((n-2)|(2))]\\
&=\left \{\begin{array}{ll} [Y^{(n-2,1^2)}]&\text{if $p\mid n$,}\\ {[Y^{(n-1,1)}]}+{[Y^{(n-2,1^2)}]}&\text{if $p\nmid n$.}\end{array}\right .
\end{align*} Since $M((n-2,2)|\varnothing)=M^{(n-2,2)}$ and $M((n-2,1,1)|\varnothing)=M^{(n-2,1,1)}$, using \cite[Corollary 3.5]{AHenke} for the module $M^{(n-2,2)}$, we obtain our desired result.
\end{proof}

Our second result offers some explicit signed $p$-Kostka numbers.

\begin{prop}\label{P: signed Kostka} Let $m,b,r\in\NN_0$ such that $1\leq r\leq p-1$. and $n=m+bp+r$.
\begin{enumerate}
  \item [(i)] If $b\geq 1$ then $k_{((m,1)|p(b)),((m)\cup(2,1^{p-1})|p(b-1))}=1$.
  \item [(ii)] If $r\geq 2$ then $k_{((m,1^r)|p(b)),((m)\cup(2,1^{r-2})|p(b))}=r-1$.
  \item [(iii)] We have \[k_{((m,1^r)|p(b)),((m+1,1^{r-1})|p(b))}=\left \{\begin{array}{ll} r-1&\text{if $p\mid n$,}\\  r&\text{if $p\nmid n$ and $p\nmid m$,}\\  1&\text{if $p\nmid n$ and $p\mid m$.} \end{array}\right .\]
\end{enumerate}
\end{prop}
\begin{proof} Let $k=bp+r$ so that $n=m+k$. By Lemma \ref{L: simplify}, we have
\begin{align}\label{Eq: 4}
  &\sum_{j=0}^b\sum_{\xi\in\P(k-jp)}(-1)^{|\xi|-\ell(\xi)} c_\xi^{(p)}[M((m)\cup\xi|p(j))]\notag\\
  =&\left \{\begin{array}{ll} [Y((m,1^r)|p(b))]&\text{if $p\mid n$,}\\ {[Y((m,1^r)|p(b))]}+{[Y((m+1,1^{r-1})|p(b))]}&\text{if $p\nmid n$.}\end{array}\right .
\end{align} We split into 2 cases which Case (A): $p\mid n$ and Case (B): $p\nmid n$. Calculations for case (B) is almost identical with case (A) and will be left to the reader.

Case (A): Suppose first that $r\geq 2$. Let $\zeta=(2,1^{r-2})$. The element $[M((m)\cup \zeta|p(b))]$ appears in the summation of Equation \ref{Eq: 4} with coefficient $(-1)^{r-(r-1)}c_\zeta^{(p)}=-(r-1)$. Notice that if $((m)\cup \zeta|p(b))\unrhd ((m)\cup\xi|p(j))$ then $j=b$ and $\xi\in\{(1^r),\zeta\}$. Since the term $[Y((m)\cup\zeta|p(b))]$ does not appear on the other side, the $-(r-1)$ copies of $[Y((m)\cup\zeta|p(b))]$ contributed by $-(r-1)[M((m)\cup\zeta|p(b))]$ must be cancelled out solely by $[M((m,1^r)|p(b))]$, i.e., \[k_{((m,1^r)|p(b)),((m)\cup(2,1^{r-2})|p(b))}=r-1.\] This proves part (ii) when $p\mid n$. 

Suppose now that $r=1$ and $b\geq 1$. The calculation is similar to the earlier case by taking $\zeta=(2,1^{p-1})$. It turns out that $(-1)^{p+1-p}c_\zeta^{(p)}=-1$ and $c^{(p)}_{(1^{p+1})}=0$. Therefore, \[k_{((m,1)|p(b)),((m)\cup(2,1^{p-1})|p(b-1))}=1.\] This proves part (i) when $p\mid n$.

For part (iii), suppose first that $r\geq 2$. Notice that, if $((m+1,1^{r-1})|p(b))\unrhd ((m)\cup\xi|p(j))$ then $j=b$ and $\xi\in\{(1^r),(2,1^{r-2})\}$. Since $[Y((m+1,1^{r-1})|p(b))]$ does not appear in Equation \ref{Eq: 4}, the contribution of $[Y((m+1,1^{r-1})|p(b))]$ by $-(r-1)[M((m,2,1^{r-2})|p(b))]$ and $[M((m,1^r)|p(b))]$ must be zero and hence \[k_{((m,1^r)|p(b)),((m+1,1^{r-1})|p(b))}=(r-1)k_{((m,2,1^{r-2})|p(b)),((m+1,1^{r-1})|p(b))}.\] Let $m=cp+s$ where $0\leq s\leq p-1$. Since $p\mid n$ and $r\geq 2$, we have $1\leq s\leq p-2$. By Lemma \ref{L: Kostka reduction}, \cite[Corollary 1.1]{BowGia} and \cite[Corollary 3.5]{AHenke}, we obtain
\begin{align*}
   k_{((m,2,1^{r-2})|p(b)),((m+1,1^{r-1})|p(b))}=&k_{((s,2,1^{r-2})|\varnothing),(s+1,1^{r-1})|\varnothing)}\\
   =&k_{(s,2,1^{r-2}),(s+1,1^{r-1})}\\
   =&k_{(s,2),(s+1,1)}k_{(1^{r-2}),(1^{r-2})}=1.
\end{align*} Therefore $k_{((m,1^r)|p(b)),((m+1,1^{r-1})|p(b))}=(r-1)\cdot 1=r-1$. Suppose now that $r=1$. By Lemma \ref{L: simplify}, since $(m+1)=\varnothing+p(c+1)$ and $m<p(c+1)$, we have \[k_{((m,1)|p(b)),((m+1)|p(b))}=0=r-1.\] This proves part (iii) when $p\mid n$.

Case (B): The calculations are similar for parts (i) and (ii). For the proof of part (iii), suppose first that $r=1$. If $m=cp+s$ then $0\leq s\leq p-2$. By Lemma \ref{L: simplify}, \[k_{((m,1)|p(b)),((m+1)|p(b))}=k_{((s,1)|\varnothing),((s+1)|\varnothing)}=k_{(s,1),(s+1)}=1\] where the final equation is obtained using \cite[Corollary 13.14]{GJ} in the semisimple case. Suppose now that $r\geq 2$, compared with case (A), instead, we have \[k_{((m,1^r)|p(b)),((m+1,1^{r-1})|p(b))}=1+(r-1)k_{((m,2,1^{r-2})|p(b)),((m+1,1^{r-1})|p(b))}.\] Let $m=cp+s$ where $0\leq s\leq p-1$. Similarly, we get
\begin{align*}
   k_{((m,2,1^{r-2})|p(b)),((m+1,1^{r-1})|p(b))}=&k_{((s,2,1^{r-2})|\varnothing),(s+1,1^{r-1})|\varnothing)}\\
   =&k_{(s,2,1^{r-2}),(s+1,1^{r-1})}\\
   =&\left \{\begin{array}{ll}k_{(s,2),(s+1,1)}k_{(1^{r-2}),(1^{r-2})}&\text{if $s\neq 0$,}\\ k_{(2,1^{r-2}),(1^r)}&\text{if $s=0$,}\end{array}\right .\\
   =&\left \{\begin{array}{ll}1&\text{if $s\neq 0$,}\\ 0&\text{if $s=0$.}\end{array}\right .
\end{align*}
\end{proof}

\end{document}